\documentclass[]{amsart}
\usepackage{latexsym,fancyhdr,amssymb,color,amsmath,amsthm,graphicx,listings,comment}
\newtheorem{thm}{Theorem} \newtheorem{lemma}{Lemma}  \newtheorem{propo}{Proposition}
\let\paragraph\subsection

\title{The Curvature of Graph Products}
\author{Oliver Knill}
\date{July 18, 2021}
\address{Department of Mathematics \\ Harvard University \\ Cambridge, MA, 02138 }

\subjclass{05CXX,  
           68R10}  
\keywords{Graph theory, Arithmetic}

\begin{document}
\maketitle

\begin{abstract}
We show that the curvature $K_{G*H}(x,y)$ at a point $(x,y)$ in the 
strong product $G*H$ of two finite simple graphs is equal to 
the product $K_G(x) K_H(y)$ of the curvatures.
\end{abstract}

\section{The product formula}

\paragraph{}
If $G$ is a finite simple graph with 
{\bf simplex generating function} $f_G(t) = 1+f_0 t + f_1 t^2 + \cdots + f_d t^{d+1}$, where $f_k$ 
counts the number of {\bf complete sub-graphs} of $G$ of dimension $k$, the
{\bf curvature} of a vertex $x$ is defined as $K_G(x) = \int_{-1}^0 f_{S(x)}(s) \; ds$,
where $S(x)$ is the unit sphere of $x$, the graph generated by the vertices directly adjacent to $x$. 
Integrating out, the curvature is 
$$ K(x) = 1-\frac{f_0(S(x))}{2}+\frac{f_1(S(x))}{3}-\frac{f_2(S(x))}{4} + \cdots  \; . $$
The {\bf Gauss Bonnet formula} \cite{cherngaussbonnet} is $\chi(G) = \sum_x K(x)$. The relation appeared
first in \cite{levitt1992}, we explored it in \cite{cherngaussbonnet} 
also in Gauss-Bonnet and in particular in a discrete manifold situations. The general Gauss-Bonnet is 
proven by distributing each term $\chi(G)=\sum_{k} \omega(k)$ with $\omega(k)=(-1)^{{\rm dim}(k)}$ from a simplex 
$k$ with ${\dim}(k)+1$ vertices to each vertex in $x$, giving each weight ${\rm dim}(k)+1$.
This Euler handshake argument works for any multi-linear valuation \cite{valuation}.

\paragraph{}
If $G,H$ are two graphs, the {\bf strong product} $G*H$ of $G$ and $H$ has as vertices the Cartesian product
$V(G) \times V(H)$ of $V(G)$ and $V(H)$. Two vertices $(a,b),(c,d)$ in
$G*H$ are connected if both projections onto $G,H$ are either an edge in $G,H$ or a
vertex in $G,H$. Together with the discrete union, augmented to an additive group, 
we get a commutative and associative ring, the {\bf Shannon ring} of graphs. All quantities
like Euler characteristic, curvature or later also the indices $i_G(x)$ are extended to this ring
with $\chi(-A)=\chi(A), K_{-G}(x)=-K_G(x)$. Our result is:

\begin{thm}
$K_{G*H}(x,y) = K_G(x) K_H(y)$.
\label{Theorem1}
\end{thm}

\paragraph{}
Since no simple relations between simplex generating functions of $G,H$ and $G*H$ exist,
we have not managed yet o get a direct combinatorial proof of this identity. For example, if $G$ is 
a complete graph $K_4$ with $f$-vector $(4,6,4,1)$ and $H$ is a star graph with $f$-vector $(5,4)$, then 
the strong product has the $f$-vector $(20, 94, 212, 277, 224, 112, 32, 4)$. 
Instead, we will show the identity {\bf integral geometrically}. Curvature is also an expectation of
{\bf Poincar\'e-Hopf indices} $i_{G,g}(x)$ \cite{poincarehopf,MorePoincareHopf}.
It turns out then that the indices 
$i_{G,g}(x), i_{H,h}(y)$ multiply and give $i_{G*H,g*h}(x,y)$. 
Taking expectations proves the curvature relation. 

\paragraph{}
If $g(x)$ is a scalar function on vertices of $G$ which is {\bf locally injective} (=coloring),
meaning $g(x) \neq g(y)$ if $x,y$ are adjacent in $G$, then the Poincar\'e-Hopf
index of $g$ at $x$ is $i_g(x) = 1-\chi(S_g(x))$, where $S_g(x) = \{ y \in S(x), g(y)<g(x) \}$.
The Poincar\'e-Hopf formula is $\chi(G) = \sum_x i_g(x)$. Similarly as in the Gauss-Bonnet formula, this can 
be proven by pushing the value $\omega(x)$ to the vertex set. But unlike in Gauss-Bonnet, where it 
is distributed equally, it is pushed onto the vertex in the complete subgraph graph $x$, where $g$ is maximal. 
For any probability space $(\Omega,\mathcal{A},P)$ of locally injective functions, the {\bf index expectation} 
$K(x) = {\rm E}[i_g(x)]$ is a curvature. In particular, if $(\Omega,P)= (\prod_x [-1,1], \prod_x (dx/2))$ 
is a product probability space, then the index expectation is the curvature defined above. 

\paragraph{}
An immediate consequence of the product formula is that the {\bf Euler characteristic}
of strong products multiply, a fact which can also be verified using cohomology:
If $p_G(t) = \sum_{k=0}^d b_k(G) t^k$ is the Poincar\'e polynomial of $G$ encoding the {\bf Betti numbers}
$b_k(G)$ then $\chi(G)=p_G(-1)$ is the Euler characteristic by Poincar\'e-Hopf. And the
K\"unneth formula implies that $G \to p(G)$ is a ring homomorphism from the Shannon ring to 
an integer polynomial ring. Of course, this assumes that cohomology is also extended naturally to negative
graphs  $b_k(-G)=-b_k(G)$. The K\"unneth formula can be verified quite readily by seeing $b_k$
as the dimension of the kernel of the Hodge Laplacian restricted to $k$-forms. In the strong product
the cup product of cohomology is swiftly implemented on the harmonic forms: if $f$ is a harmonic  $p$-form
and $g$ is a harmonic $q$ form, then $d^* f g$ is a harmonic $(p+q)$-form. 
We could get from \cite{ComplexesGraphsProductsShannonCapacity}
to \cite{KnillKuenneth} without chain homotopy but replace it with the fact that for two arbitrary
simplices, the connection graph of $G \times H$ and the Barycentric graph of $G \times H$ are 
homotopic.

\paragraph{}
We can also compare the product formula Theorem~\ref{Theorem1} in the context of 
discrete manifolds. Given two simplicial complexes $G,H$ which are discrete
manifolds, then $G \times H$ is a set of sets but not a simplicial complex.
We can  associate to $G$ a connection graph $G'$ and to $H$ a connection graph $H'$
and to $G \times H$ a connection graph $(G \times H)' = G' \star H'$. So, we have
for any simplicial complex a natural curvature at every set which has a graph theoretical
interpretation and that the curvature $K(x,y) = K(x) K(y)$ again has a graph theoretical
interpretation.

\paragraph{}
The theorem can be extended. This becomes clear when one looks at the proof. 
One way to extend it is to use curvatures defined by arbitrary probability 
spaces $\Omega_G$ and $\Omega_H$ on functions on $G,H$ and take the product probability space 
$\Omega_G \times \Omega_H$ for $G*H$. 
The curvature $K(x)$ applies for the Euler characteristic $\sum_x \omega(x)$ of $G$, where
$\omega(x)=(-1)^{{\rm dim}(x)}$ and the sum is over all complete sub-graphs $x$ of $G$. 

\paragraph{}
Limitations appear when leaving properties which are homotopy invariant.
There is also a curvature for the {\bf Wu characteristic} 
$$  \omega(G) = \sum_{x \sim y} \omega(x) \omega(y)  = \sum_x K(x) \; , $$ 
where the sum to the left is taken over all intersecting simplices $x,y$ written as $x \sim y$ for having a non-empty
intersection. But the product formula for the Wu curvature does not 
hold any more. This is no surprise as unlike Euler characteristic, Wu characteristic
is not a homotopy invariant but this is of advantage as it leads to a cohomology
which can distinguish homotopic but not equivalent fiber bundles \cite{CohomologyWuCharacteristic}. 

\section{The proof} 

\paragraph{}
It might be that a direct combinatorial check is difficult because we do not have
a simple relation between the $f$-vectors $f_G$ and $f_H$ or $f$-functions $f_G(t)$ and $g_G(t)$. 
Similarly as when we struggled 10 years ago to prove that the curvature
of discrete $(2d-1)$-dimensional manifolds is constant zero (this is not even defined in the continuum
as the Gauss-Bonnet-Chern formula only applies for even dimensional manifolds), we will use 
{\bf integral geometric tools}. In the case of odd-dimensional discrete manifolds, we found first
an integral geometric proof which could also be done by Dehn-Sommerville considerations \cite{dehnsommervillegaussbonnet}.
Here however, we deal with a result which holds for all finite simple graphs. Integral geometry
\cite{colorcurvature, indexexpectation} might remain as the most elegant approach. 

\paragraph{}
The key is to show that the unit sphere $S(x,y)$ at a point $(x,y)$ of $G \times H$ is homotopic to 
the join $S(x) \oplus S(y)$. For the join, of two graphs, the simplex generating functions multiply.
We need however to go through some homotopy deformation of the unit spheres $S(x,y)$ to see them as
the join of $S(x)$ and $S(y)$. This pretty much restricts the result to quantities like Euler characteristic
which are homotopy invariant. Now as has been known also since half a century, Euler characteristic is
the only linear valuation \cite{valuation} on which has this property. 
(Homotopy invariance implies invariance under Barycentric refinements
for which we know the Barycentric refinement operator explicitly and have it only one eigenvalue $1$ which 
corresponds an eigenvector corresponding to the Euler characteristic.)

\paragraph{}
The actually formula which holds for any finite simple graphs $G,H$ with vertices $x \in V(G)$ and $y \in V(H)$.
Let $B(x)$ the unit ball of $x$, the graph generated by the union of $S(x)$ and $x$.  

\begin{lemma}[Cylinder relation]
The unit sphere $S(x,y)$ of a vertex $(x,y)$ in $G*H$ satisfies
$S(x,y) = S(x)*B(y) \cup B(x)*S(y)$ with intersection set $S(x)*S(y)$.
\end{lemma}
\begin{proof}
This is a direct consequence of the definition of the product.
Given a point $(a,b)$ in $G*H$ which is in $S(x,y)$, then
necessarily, $a$ is in $B(x)$ and $b$ is in $B(y)$. But since 
$(a,b)$ is different from $(x,y)$, we either $a \neq x$ which means
we are in $S(x)*B(y))$ or then $b \neq y$ which means we are
in $B(x)*S(y)$. 
\end{proof} 

\begin{lemma}[Cylinder to Join]
$S(x)*B(y) \cup B(x)*S(y)$ is homotopic to $S(x) \oplus S(y)$, where
$\oplus$ is the Zykov join. 
\end{lemma}
\begin{proof}
Every unit ball is a ball (defined as a sphere in which a point is removed)
and so contractible. Since $B(x)$ and $B(y)$ are contractible, we can deform each of them 
to a point. What end up with a situation, where
every point of $S(x)$ is connected to every point of $S(y)$. 
\end{proof}

\paragraph{}
A picture to visualize is to see $S(x,y)$ as the boundary of a
closed cylindrical can. In that case, $S(x)$ is a $1$-circle 
and $S(y)$ is a $0$-circle consisting of two points. 
The set $S(x)*B(y)$ is the cylindrical mantle of the can. 
The set $B(x)*S(y)$ are the top and bottom lids to the can. 
To make the deformation, deform the height of the can to zero.
Lift the center of the top lid and the center of the bottom lid
to get a union of two lids. Now deforming $B(y)$ to $S(y)$
will lead to a graph consisting of two unions of $S(x)$ and $S(y)$
where each point in $S(y)$ is connected to every point in $S(x)$. 

\paragraph{}
Let us write $\oplus$ for the join. 
From the ``cylinder to join" lemma we can assume that $S(x,y) = S(x) \oplus S(y)$. 
But the same also applies to the stable part $S_{gh}^-(x,y) = S_g^-(x) \oplus S_h^-(y)$. 
But this implies then that the indices tensor multiply. We started to work 
on the simplex generating function in 
\cite{AverageSimplexCardinality}. (It is just a sort of shifted Euler polynomial.)

\begin{propo}
$i_{G*H,g h}(x,y) = i_{G,g}(x) i_{H,h}(y)$. 
\end{propo}

\begin{proof}
In general, for any graphs $G,H$, we have $(1-\chi(G)) (1-\chi(H)) = (1-\chi(G \oplus H))$,
where $G \oplus H$ is the join. This simply follows from the fact that if $x$ is a $(k-1)$-simplex in $G$
and $y$ is a $(l-1)$-simplex in $H$, then $(x+y)$ is a $(k+l-1)$-simplex in $G+H$. 
The formula $(1-\chi(S_{G*H,g*h}(x)))= (1-\chi(S_{G,g}(x))) (1-\chi(S_{H,h}(x)))$
follows now from $S_{G*H,g*h}(x,y)$ being homotopic to the join of 
$S_{G,g}(x)$ and $S_{G,h}(y)$ and that the Euler characteristic is a homotopy invariant. 
We have $S_{G*H}^-(x,y) = B_G^-(x)*S_H^-(y) \cup S_G^-(x)*B_H^-(y)$,
which is homotopic to the join of $S_H^-$ and $S_G^-$.
\end{proof} 

\paragraph{}
The proof of the theorem is now straight forward: 

\begin{proof}
Take the expectation of the relation given in the Proposition 
and make use of the fact that the two random variables $X(g)=i_{G,g}(x)$ and 
$Y(h)=i_{H,h}(y)$ are independent. The expectation of the product is the product of the 
expectations. Here is a bit more formal derivation: \\
a) The random variables $g \to X(g), h \to Y(y)$ are independent. \\
b) $K(x) = {\rm E}[X] ,  K(y) = {\rm E}[Y]$ \\
c) $K(x,y) = {\rm E}[X Y]  = {\rm E}[X] {\rm E}[Y]$. 
\end{proof}

\section{Remarks}

\paragraph{}
The {\bf Lefschetz fixed point theorem} for graphs or simplicial complexes \cite{brouwergraph} assures that the 
{\bf Lefschetz number} $\chi(G,T)$ (the super trace of $T$ induced on harmonic forms) of 
a graph endomorphism $T: G \to G$ is equal to the sum $\sum_{x \in F} i_T(x)$ of {\bf indices}
of all fixed simplices $T(x)=x$. The index is $\omega(x=(-1)^{{\rm dim}(x)} {\rm sign}(T:x \to x)$.
A special case is the {\bf Brouwer fixed point theorem}: if $G$ is homotopic to $1$ as then, the Lefschetz 
number is $1$ so that there is at least one simplex $x$ that is fixed. An other special case
is if $T:G \to G$ is the identity, where the Lefschetz number is the Euler characteristic and
$i_T(x)=\omega(x)$. In that case, $\chi(G) = \sum_{x} \omega(x)$ is the definition of Euler characteristic. 
It is a bit easier to see that the Lefschetz number is compatible with the Shannon product. In the 
next section, we show why. The key reason is that the cup product implementation on cohomology is by Hodge
directly implementable using harmonic forms. As Whitney also realized
there is nothing mysterious about the cup product in cohomology: $k$-forms are functions
on complete sub-graphs with $(k+1)$-elements. The Hodge Laplacian is block diagonal and induces maps on the
finite dimensional vector spaces of $k$ forms. Now, if $g$ is a harmonic $p$ form and $h$ is a 
harmonic $q$ form, then the tensor product $g(x) h(y)$ is a function on complete $p+1+q+1$ graphs but
$d^* g h$ is now a harmonic $(p+q)$-form. This will imply that the Lefschetz number is compatible
with the Shannon product.

\paragraph{}
Given two finite simple graphs $G,H$ and two graph endomorphisms $T:G \to G, S: H \to H$. Then this induces
a natural endomorphism on $G*H$ as follows: if $T: V(G) \to V(G)$ and $S: V(H) \to V(H)$ 
are the induced maps on vertex sets: then the map $T*S(v,w) = (T(v),S(w))$ is a map on vertices of $G*H$ and
is again a graph endomorphism on $G*H$. The fixed point sum $\sum_{x=T*S(z)} i_{T*S}(z)$ is equal to 
$(\sum_{x=T(x)} i_T(x)) (\sum_{y=S(y)} i_S(y)$. The reason is that every fixed simplex $z$ of $T*S$ becomes after
projection onto $G$ or $H$ a fixed point of $T$ or $S$ respectively.
Also the Lefschetz number is multiplicative $\chi(G*H,T*S) = \chi(G,T) \chi(H,S)$, as one can see by 
diagonalizing each $T: H^k(G) \to H^k(G)$ and then see that if $T: g_k = \lambda_k g_k$ on $H^k(G)$
and $S: h_l = \mu_l h_l$ then $T*S: d^* g_k h_l = \lambda_k \mu_l d^* g_k h_l$. If $g_k$ was a $p$-form
and $h_l$ was a $q$-form, then $d^* g_k h_l$ is $(p+q)$-form. The super traces of $T,S$ induced on 
cohomology (kernels of Hodge Laplacians on $p$-forms) therefore multiply. If 
$L(G,T) = \sum_{p=0} (-1)^p {\rm tr}(T|H^p(G) \to H^p(G)) = \sum_{p} (-1)^p \sum_k \lambda_{k,p}$ and
$L(H,S) = \sum_{q=0} (-1)^q {\rm tr}(T|H^q(H) \to H^q(H)) = \sum_{q} (-1)^q \sum_l \mu_{l,q}$, then 
$L(G*H,T*S) = \sum_{p,q} (-1)^{p+q} \sum_{k,l} \lambda_{k,p} \mu_{l,q}$. 

\paragraph{}
The Cartesian product of two even-dimensional Riemannian manifolds $M \times N$ is a manifold
which has Gauss-Bonnet-Chern curvature $K_{M \times N}(x,y) = K_M(x) K_N(y)$, the product
of the curvatures of the individual parts. The curvature $K_M(x)$ is the
index expectation ${\rm E}[i_{M,g}(x)]$ of Poincar\'e-Hopf indices of Morse functions $g$ obtained
by Nash-embedding $M$ into an ambient Euclidean space and restricting a linear function to $M$.
The product formula for curvature could then be derived from the fact that $g \to i_{M,g}(x)$
and $g \to i_{N,g}(y)$ are independent random variables in the product probability space.
This approach gives a Gauss-Bonnet formula for any probability space on Morse functions. The Gauss-Bonnet
Chern version is a special symmetric form in which the probability space of Morse functions is rotationally
symmetric. 

\paragraph{}
The probabilistic approaches to curvature does not require the curvature to be smooth at all in the 
manifold case. It could be a divisor for example, like the Poincar\'e-Hopf indices themselves, which correspond
to probability spaces with one point only. If $\Omega(M),\Omega(N)$ are probability spaces and $i_g(G),i_h(G)$
divisor-valued random variables, then $i_g(G),i_H(G)$ are independent (they become independent random variables
after applying a test function: take a finite subset $U$ of vertices and sum up the $i_g(G)$ values over $U$).
In classical Riemannian geometry, one can derive the product formula also from the fact that the
Riemann curvature tensor $R_{ijkl}(x,y)$ at a point $(x,y)$ is block diagonal, if the basis in $T_{x,y} M \times N$
is adapted to the product because the sectional curvatures $R_{ijkl}(x,y)$ are then $0$ if $e_i \in T_xM$ and
$e_j$ in $T_yN$ and the basis is an orthonormal basis. We once explored the possibility to 
to deform a positive curvature manifold so that its index expectation curvature which is positive. 
\cite{DiscreteHopf,DiscreteHopf2} were attempts in this direction.

\paragraph{}
Assume now that $M,N$ are discrete $p$ and $q$-manifolds, (where with a discrete $p$-manifold we mean a 
finite simple graph for which all of the unit spheres is a $(p-1)$-sphere),
we have a Cartesian product in which pairs $(x,y)$ are vertices, where $x$ is a complete sub-graph
of $M$ and $N$ is a complete subgraph of $N$ and two points $(a,b),(c,d)$ are connected if either
$a \subset c, b \subset d$ or $a \supset c, b \supset d$. This product is related to the 
{\bf Stanley-Reisner product}, when a complex is represented as a polynomial ring and has the property that
$M \times N$ is still a manifold. However, the curvature is not the product of the curvatures of
$M$ and $N$. The product is also not associative. The only initial drawback of the strong product of
two manifolds is no more a manifold. This is not a problem
but an opportunity to widen up the notion of discrete d-manifold.
One possibility is to require that every unit sphere is homotopic to a $(d-1)$-sphere, an other is to allow
also unit-spheres which are contractible. We will write about this in some other work.

\paragraph{}
The curvature value $K(x)$ is defined for all finite simple graphs $G$ and as in the continuum, the curvature
is the index expectation of Poincar\'e-Hopf indices $i_{G,g}$. This is all very general,
holding for all finite simple graphs. The continuum in the form of compact Riemannian manifold $M$
can be seen as a limiting case. Given $h>0$, and a finite dense enough set of points $V$ in $M$
we get a graph by connect points which are close enough $d(x,y)<h$. For example, if $n=e^{1/h}$ are
chosen independently then the discrete curvature averaged over a ball of radius $\sqrt{1/h}$
converges for $h \to 0$ to the classical curvature. We will have to give here more details but the
reason is that both the finite graph as well as the compact manifold can both be embedded isometrically
into a large dimensional Euclidean space $E$. Take in both cases the same probability space of linear 
functions in $E$ (which carries a natural Haar measure). The Poincar\'e-Hopf indices of $G$ and $M$
now are very close. In the graph case, the expectation gives the Levit curvature we deal with.
In the manifold case, the expectation has to produce Gauss-Bonnet-Chern. 

\paragraph{}
The strong product of Shannon \cite{Shannon1956} was not introduced in a geometric setting at first
but as a tool in communication theory. It leads to a nice ring \cite{ArithmeticGraphs} which even can be
topologically completed. It shows remarkable compatibility both on topological and spectral level
\cite{RemarksArithmeticGraphs}. The strong product $G*H$ is often homotopic to a manifold
\cite{ComplexesGraphsProductsShannonCapacity} if $G$, $H$ are manifolds. This means that $G*H$
is a generalized manifold in which every unit sphere is homotopic to a sphere.
The Shannon product is compatible with cohomology which is visible from the fact that the
map mapping $G$ to the Poincar\'e-polynomial $p_G(t) = \sum_{k} b_k(G) t^k$ is a ring homomorphism
from the ring $(\mathcal{Z},+,*,0,1)$ of graphs to the ring $\mathbb{Z}[t]$ of integer polynomials.
This requires to extend cohomology to negative graphs as $b_k(-G) = - b_k(G)$. In any case, 
we have a remarkable compatibility of the Shannon ring both with topology as well as 
differential geometry. 

\section{Code}

\paragraph{}
The following code computes the curvature, as well as the Poincar\'e-Hopf indices with respect 
to some random function. 

\fontsize{1}{1} \selectfont
\lstset{language=Mathematica} \lstset{frameround=fttt}
\begin{lstlisting}[frame=single]
Generate[A_]:=Sort[Delete[Union[Sort[Flatten[Map[Subsets,A],1]]],1]]; 
Whitney[G_]:=Generate[FindClique[G, Infinity, All]]; NH=NeighborhoodGraph;
Fvector[G_]:=If[Length[VertexList[G]]==0,{},Delete[BinCounts[Map[Length,Whitney[G]]],1]];
Ffunction[G_,t_]:=Module[{f=Fvector[G]},1+Sum[f[[k]] t^k,{k,Length[f]}]]; 
EulerChi[G_]:=Module[{f=Fvector[G]},Sum[f[[k]](-1)^(k+1),{k,Length[f]}]]; 
S[G_,x_]:=Module[{H=NH[G,x]},If[Length[VertexList[H]]<2,Graph[{}],VertexDelete[H,x]]];
Curvature[G_,v_]:=Module[{G1,vl1,n1,k,u},G1=S[G,v];vl1=VertexList[G1]; n1=Length[vl1];
  u=Fvector[G1]; 1+Sum[u[[k]]*(-1)^k/(k+1),{k,Length[u]}]];
Curvatures[G_]:=Module[{vl=VertexList[G]},Table[Curvature[G,vl[[k]]],{k,Length[vl]}]];
RandFunction[G_]:=Table[Random[],{Length[VertexList[G]]}]; 
Indices[G_,g_]:=Module[{v=VertexList[G],w,H,T,u}, 
 Table[ T=S[G,v[[k]]]; u=VertexList[T]; w={};
        Do[If[g[[j]]<g[[k]] && MemberQ[u,v[[j]]],w=Append[w,v[[j]]]],{j,Length[v]}];
        H=Subgraph[T,w]; 1-EulerChi[H],{k,Length[v]}]]

StrongProduct[G_,H_]:=Module[{v,e={},e1,e2,q,
                  vG=VertexList[G],vH=VertexList[H],eG=EdgeList[G],eH=EdgeList[H]},
  eG=Table[Sort[{eG[[k,1]],eG[[k,2]]}],{k,Length[eG]}];
  eH=Table[Sort[{eH[[k,1]],eH[[k,2]]}],{k,Length[eH]}];
  v=Partition[Flatten[Table[{vG[[k]],vH[[l]]},{k,Length[vG]},{l,Length[vH]}]],2];
  Do[If[v[[k,2]]==v[[l,2]]&&MemberQ[eG,Sort[{v[[k,1]],v[[l,1]]}]],
         e=Append[e,v[[k]]->v[[l]]]],{k,Length[v]},{l,Length[v]}];
  Do[If[v[[k,1]]==v[[l,1]]&&MemberQ[eH,Sort[{v[[k,2]],v[[l,2]]}]],
         e=Append[e,v[[k]]->v[[l]]]],{k,Length[v]},{l,Length[v]}];
  e1=Table[{eG[[k,1]],eH[[l,1]]}->{eG[[k,2]],eH[[l,2]]},{k,Length[eG]},{l,Length[eH]}];
  e2=Table[{eG[[k,1]],eH[[l,2]]}->{eG[[k,2]],eH[[l,1]]},{k,Length[eG]},{l,Length[eH]}];
  q=Flatten[Union[e,e1,e2]]; UndirectedGraph[Graph[v,q]]];
TensProduct[g_,h_]:=Flatten[Table[g[[k]] h[[l]],{k,Length[g]},{l,Length[h]}]]; 

G=RandomGraph[{8,12}]; g=RandFunction[G];   KG=Curvatures[G]; IG=Indices[G,g]; 
H=RandomGraph[{7,17}]; h=RandFunction[H];   KH=Curvatures[H]; IH=Indices[H,h]; 
GH=StrongProduct[G,H]; gh=TensProduct[g,h]; KK=Curvatures[GH];II=Indices[GH,gh]; 

{TensProduct[IG,IH]==II,  TensProduct[KG,KH]==KK, Total[KK]==Total[II]==EulerChi[GH]}
\end{lstlisting}
\fontsize{8}{8} \selectfont

When increasing parameters, finding the list of all 
complete sub-graphs can be hard to compute. Note that as usual, grabbing code from the PDF
does not work. But the source code can as usual be copy pasted from the LaTeX source 
on the ArXiv: 

\paragraph{}
The following code does compute the simplex generating function and so also the number of all 
complete sub-graphs for any $k$ recursively using Gauss-Bonnet. While the above version is faster,
the following computation is timeless and does not invoke any clique computations. It is elegant
but in general slow. It could allow however to push the threshold of possible computations larger. 
Assume for example you have a graph for which your favorite clique finding algorithm finds the
clique generating function for all unit spheres, then one can get the clique generating function of
$G$. 

\fontsize{2}{2} \selectfont 
\lstset{language=Mathematica} \lstset{frameround=fttt}
\begin{lstlisting}[frame=single]
NH=NeighborhoodGraph; f[G_,t_]:=Module[{v=VertexList[G],e=EdgeList[G],n,w},n=Length[v];
  If[e=={},1+n*t,1+Integrate[Sum[w=v[[k]];f[VertexDelete[NH[G,w],w],s],{k,n}],{s,0,t}]]];
S[G_,x_]:=Module[{H=NH[G,x]},If[Length[VertexList[H]]<2,Graph[{}],VertexDelete[H,x]]];
Curvature[G_,x_]:=Module[{f0=f[S[G,x],t]}, Integrate[f0,{t,-1,0}]];  
Curvatures[G_]:=Module[{V=VertexList[G]},Table[Curvature[G,V[[k]]],{k,Length[V]}]];

G=RandomGraph[{30,100}]; f0=f[G,t]; EulerChi = 1-f0 /. t->-1; K=Curvatures[G];
Print[EulerChi==Total[K]]; Print[K]; 
\end{lstlisting}
\fontsize{8}{8} \selectfont

\begin{figure}[!htpb]
\scalebox{1.0}{\includegraphics{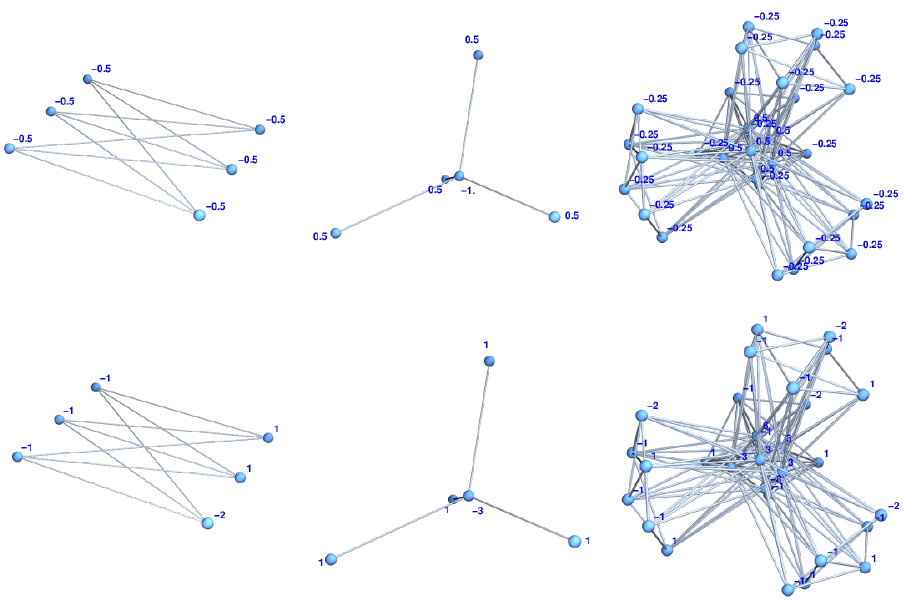}}
\label{Example}
\caption{
An example of two graphs. Above we see the curvatures, below the indices
taken with a random function.
}
\end{figure}

\bibliographystyle{plain}

\end{document}